%% file: cr_mod2.tex
\documentclass{article}
\usepackage[american]{babel}
\usepackage{tikz}
\usepackage{amsfonts,amsmath,amssymb,epsf,epsfig,stmaryrd}
\usepackage{amsthm}
\usepackage{xypic}
\xyoption{all}
\input{xypic}
\newdir{ >}{{}*!/-8pt/\dir{>}}

\def\Z{\mathbb{Z}}

   \def\id{{\rm id}}

\newtheorem{theo}{Theorem}[section]
\newtheorem{Defi}[theo]{{\bf Definition}}
\newenvironment{defi}{\begin{Defi} \normalfont}{\end{Defi}}

\newtheorem{Prop}[theo]{{\bf Proposition}}
\newenvironment{prop}{\begin{Prop} \normalfont}{\end{Prop}}
\newtheorem{Cor}[theo]{{\bf Corollary}}
\newenvironment{cor}{\begin{Cor} \normalfont}{\end{Cor}}
\newtheorem{Lem}[theo]{{\bf Lemma}}
\newenvironment{lem}{\begin{Lem} \normalfont}{\end{Lem}}

\newtheorem{Exa}[theo]{{\bf Example}}
\newenvironment{exa}{\begin{Exa} \normalfont}{\end{Exa}}  

\newtheorem{Rem}[theo]{{\bf Remark}}
\newenvironment{rem}{\begin{Rem} \normalfont}{\end{Rem}}

\begin{document}
\title{Crossed modules of racks} 

\author{Alissa S. Crans\\
         Mathematics Department\\
         Loyola Marymount University\\
         One LMU Drive, Suite 2700\\
         Los Angeles, CA 90045 \\
         email: acrans@lmu.edu\\
\and    Friedrich Wagemann\\
        Laboratoire de Math\'ematiques Jean Leray\\
        Universit\'e de Nantes\\
        2, rue de la Houssini\`ere\\
        44322 Nantes, cedex 3\\
        email: wagemann@math.univ-nantes.fr}

\maketitle

\begin{abstract}
 We generalize the notion of a crossed module of groups to that of a crossed module of 
racks. We investigate the relation to categorified racks, namely strict 2-racks, and 
trunk-like objects in the 
category of racks, generalizing the relation between crossed modules of groups and strict
$2$-groups. Then we explore topological applications. We show that by applying the rack-space
functor, a crossed module of racks gives rise to a covering. Our main result shows how 
the fundamental racks associated to links upstairs and downstairs in a 
covering fit together to form a crossed module of racks.   
\end{abstract}





\section{Introduction}

Racks are generalizations of groups whose axioms capture essential properties of 
group conjugation and algebraically encode two of the Reidemeister moves.  
Because of the latter, they have proven useful in defining link and knot invariants. 
We remind the reader that a {\it rack} is a set $X$ equipped with a binary 
operation $(x,y)\mapsto x\lhd y$ that is invertible and 
self-distributive, i.e. 
\[(x\lhd y)\lhd z\,=\,(x\lhd z)\lhd(y\lhd z)\]
for all $x,y,z\in X$.
The most important example of a rack structure comes from the conjugation in a group, i.e.
defining $g\lhd h:=h^{-1}gh$ for all $g,h\in G$ for a group $G$ satisfies both rack axioms.  

A {\it crossed module of groups} is a quadruple $(M,N, \mu, \cdot)$ where $M$ and $N$ are groups, 
$\mu:M\to N$ is a group homomorphism, and $\cdot$ is an action 
of $N$ on $M$ by automorphisms such that $\mu$ is {\it equivariant} (considering on $N$ 
the adjoint action)
and $\mu$ satisfies the {\it Peiffer identity}, that is, the operation $m\cdot \mu(m')$ 
is just the conjugation $(m')^{-1}mm'$  in $M$ for all $m,m'\in M$.

Crossed modules of groups were introduced by J.H.C. Whitehead \cite{W1, W2} in 1949.
Whitehead remarked that for a pointed pair of spaces $(X,A)$, the connecting homomorphism
in the long exact sequence of the pair
\[  \partial : \pi_2(X,A)\to \pi_1(A)  \]
is a crossed module of groups using the natural action of $\pi_1(A)$ on all
higher homotopy groups. Whitehead showed in particular that in the special case 
where $X$ is obtained from $A$ by attaching $2$-cells, the crossed module is free.  
Crossed modules were also used to 
express $3$-cohomology 
classes by (equivalence classes of) algebraic objects, see \cite{MacWhi}, \cite{LodKas}. 
Crossed modules have been generalized to higher dimensional crossed complexes
in order to capture higher degree homotopy, see for example Brown's work on the higher 
van Kampen Theorem \cite{Bro1}, \cite{Bro2}. 
More recently, crossed 
modules of groups have attracted renewed interest in the search for categorifications of 
algebraic notions, as they are equivalent to strict $2$-groups, i.e. category objects in
the category of groups, see \cite{Lod} or \cite{Mac}.   

In one of the foundational articles about racks Fenn and Rourke introduce
the notion of the fundamental rack of a link \cite{FenRou}, which is at the heart of 
their approach to link invariants. Given a codimension two embedding $L:M\subset Q$ in 
a connected manifold $Q$ which is framed and transversally oriented, the {\it fundamental
rack} of the link $L$ is a set of homotopy classes of paths in 
$Q_0:=\overline{Q\setminus N(M)}$, where 
$N(M)$ is a (closed) tubular neighborhood of the submanifold $M\subset Q$.   
In fact, the fundamental rack is an example of an 
{\it augmented rack}, which we will encounter later in this paper. Its associated 
crossed module of groups is Whitehead's crossed module of groups
\[\partial : \pi_2(Q,Q_0)\to \pi_1(Q_0).\]     

In this paper, we generalize the notion of a crossed module from groups
to racks. Motivation for this study comes from the relation of crossed module of groups to strict
$2$-groups on the algebraic side and from the notion of the fundamental rack of a link on the 
topological side.  

On the algebraic side, we first develop three classes of examples of crossed modules of racks, which are
(in growing generality): crossed modules of groups, augmented racks and generalized augmented racks. 
This last structure consists simply of a rack $R$, a rack module $X$ and an equivariant map 
$p:X\to R$. We begin in Section \ref{first} by reminding the reader of relevant rack 
definitions and then we introduce the notion of a crossed module of racks.  
We show that a generalized augmented rack
is equivalent to a crossed module of racks. In other words, we show that the rack 
structure on $X$ is 
encoded in the module structure. We provide examples of crossed modules of racks and 
demonstrate the relationship between the examples.

We continue in Section \ref{second}, by investigating the relation of crossed modules 
of racks to {\it strict 2-racks}, or {\it categorical
racks}, i.e. to category objects in the category of racks. The relation is, 
disappointingly, not as strong
as in the group setting. The main reason for this is that we lack the ability 
to construct a category from
a rack, a module and an equivariant map - a construction which is well-known 
in the case of groups. 

Our remedy is to consider trunk-like objects in the category of racks, which 
we explore
in Section \ref{third}. The correspondence between these objects and crossed 
modules of racks is much better behaved, as demonstrated in Proposition 
\ref{correspondence_cr_mod_racks_trunks}. 

Finally, we conclude in Section \ref{apps}, by exploring applications of 
these rack structures to knot theory.  We note that Kauffman and Martins 
\cite{KM, Martins, Martins1} have obtained knot invariants from strict 
$2$-groups considered as crossed modules.  Given the relationship between 
strict $2$-groups and strict $2$-racks, we expect that Kauffman and 
Martins' methods relate to the material presented here.   Moreover, 
it turns out that the notion of a crossed module of racks
is related to coverings in the setting of link invariants of \cite{FenRou}. 
On one hand, passing
to the associated rack spaces transforms a crossed module of racks 
into a covering, as shown in  
Proposition \ref{cr_mod_racks_to_covering}.  On the other hand,
starting from links in a covering, we can construct a crossed module 
of the corresponding fundamental 
racks as proven in Theorem \ref{covering_to_cr_mod_racks}.\\  

\noindent{\bf Acknowledgements:}\quad
FW heartily thanks Alan Weinstein for the invitation to UC Berkeley during Spring 2013. 
FW furthermore thanks Simon Covez for discussion about the 
associated group of a rack.    

\section{Crossed modules of racks} \label{first}

\subsection{Basic definitions}

We begin by recalling the notion of a rack, which, as mentioned, results from 
axiomatizing the notion of group conjugation. 

\begin{defi}
A {\bf right rack} consists of a set $X$ equipped with a binary operation denoted $(x,y)\mapsto x\lhd y$
such that for all $x, y,$ and $z\in X$, the map $x\mapsto x\lhd y$ is bijective and
\[(x\lhd y)\lhd z\,=\,(x\lhd z)\lhd(y\lhd z).\]
\end{defi}  

There is also the notion of a left rack where the operation is written $(x,y)\mapsto x\rhd y$.
The left rack operation, then, satisfies
\[x \rhd (y\rhd z)\,=\,(x\rhd y)\rhd(x\rhd z)\]
for all $x,y,$ and $z\in X$.
One can always transform a left rack into a right rack (and vice-versa) by sending the 
bijective map $y\mapsto x\rhd y$ to its inverse (which is then denoted $z\mapsto z\lhd x$). 
In this paper, we will work with right racks as in \cite{FenRou}, \cite{FRS}.  

The conjugation in a group $G$ gives rise to a (left)
rack operation given by $(g,h)\mapsto ghg^{-1}$ and a right rack operation by
$(g,h)\mapsto g\lhd h:=h^{-1}gh$. For a point of view on racks where the two operations are treated 
on an equal basis, see e.g. \cite{MarPic} p.17. 

The notion of a unit leads to pointed racks.

\begin{defi}
A {\bf pointed rack} $(X,\lhd,1)$ consists of a set $X$ equipped with a binary 
operation $\lhd$ and an element $1\in X$ 
satisfying:
\begin{enumerate}
\item $(x\lhd y)\lhd z\,=\,(x\lhd z)\lhd(y\lhd z)$ for all $x,y,z\in X$,
\item For each $a,b\in X$, there exists a unique $x\in X$ such that $x\lhd a\,=\,b$,
\item $1\lhd x\,=\,1$ and $x\lhd 1\,=\,x$ for all $x\in X$.
\end{enumerate}
\end{defi} 

Once again, the conjugation rack of a group is an example of a pointed rack. 
For formal reasons, we will denote the conjugation rack underlying the group 
$G$ by ${\rm Conj}(G)$. Denote by ${\tt Racks}$ the category of racks, i.e. 
the category whose objects 
are racks and whose morphisms are rack homomorphisms as defined below. Then, 
${\rm Conj}$ is a functor from the category of groups ${\tt Grp}$ to ${\tt Racks}$.

\begin{defi}
Let $R$ and $S$ be two racks. A {\bf morphism of racks} is a map $\mu:R\to S$ such that 
\[\mu(r\lhd r')\,=\,\mu(r)\lhd \mu(r')\]
for all $r\in R$.
In the usual way, we will speak about iso- and automorphisms of racks. 
\end{defi}

\begin{defi}
Let $R$ be a rack. The {\bf associated group to $R$}, denoted ${\rm As}(R)$, is the quotient 
of the free group $F(R)$ on the set
$R$ by the normal subgroup generated by the elements $y^{-1}x^{-1}y(x\lhd y)$ for all $x,y\in R$.
Denote the canonical morphism of racks by $i:R\to{\rm As}(R)$. 
\end{defi}

\begin{exa}
Consider the set $R:=\{x,y\}$ consisting of two elements $x$ and $y$ with an operation given by
\[ x\rhd x\,=\,y,\,\,\,\,x\rhd y\,=\,x,\,\,\,\,y\rhd y\,=\,x,\,\,\,\,y\rhd x\,=\,y. \]
It is easy to verify that $R$ is indeed a rack. The relation $x\rhd y = xyx^{-1}$ in ${\rm As}(R)$
implies that $x=y$ in ${\rm As}(R)$. It turns out that ${\rm As}(R)$ is isomorphic to $\Z$, thus
the canonical map $i:R\to{\rm As}(R)$ is not necessarily injective.
\end{exa}

The importance of the associated group ${\rm As}(R)$ of a rack $R$ comes from the following 
universality property:

\begin{lem}
Let $R$ be a rack and $G$ be a group. For any morphism of racks $f:R\to {\rm Conj}(G)$, 
there exists a unique group morphism $g:{\rm As}(R)\to G$ such that $g\circ i\,=\,f$.
\end{lem}

From this, one can deduce that the functor ${\rm As}:{\tt Racks}\to{\tt Grp}$ from the category of
racks to the category of groups is left adjoint to the functor ${\rm Conj}:{\tt Grp}\to{\tt Racks}$ 
which associates to a group its underlying conjugation rack.

\begin{rem}
In fact, the unit of the adjunction is just the map $i$. 
By standard arguments, the unit of the adjunction is injective, but only as a map
\[ i : {\rm Conj}(G)\to  {\rm Conj}({\rm As}({\rm Conj}(G)))\]
for a group $G$. 
\end{rem}

We observe that the compositions ${\rm Conj}({\rm As}(R))$ for a racks $R$ and ${\rm As}({\rm Conj}(G))$
for a group $G$ are, in general, far from being equal to $R$ or $G$ respectively.   
For example, for an abelian group $A$, the 
conjugation rack ${\rm Conj}(A)$ is the set $A$ with the trivial rack product, while 
${\rm As}({\rm Conj}(A))$ is the free abelian group on the set $A$. 
   
\begin{defi}  \label{definition_rack_action}
Let $R$ be a rack and $X$ be a set. We say that $R$ {\bf acts on} $X$ (or that $X$ is an $R$-set) when
there are bijections $(\cdot r):X\to X$ for all $r\in R$ such that 
\[(x\cdot r)\cdot r'\,=\,(x\cdot r')\cdot (r\lhd r')\] 
for all $x\in X$ and all $r,r'\in R$.
\end{defi}

\begin{lem}    \label{rack_action}
An action of $R$ on $X$ is equivalent to a morphism of racks $\mu:R\to{\rm Bij}(X)$ with values in the
conjugation rack underlying the group of bijections on $X.$
\end{lem}


\begin{defi} \label{definition_hs_product}
Let $R$ be a rack and $X$ be an $R$-set. The {\bf hemi-semi-direct product rack} 
consists of the set $X\times R$ equipped with the rack product
\[(x,r)\lhd(x',r')\,:=\,(x\cdot r',r\lhd r')\]
for all $x,x'\in X$ and all $r,r'\in R.$
\end{defi}

\begin{defi}
Let $R$ and $S$ be racks. We say that $S$ {\bf acts on $R$ by automorphisms} when there is an action
of $S$ on $R$ and 
\[(r\lhd r')\cdot s\,=\,(r\cdot s)\lhd(r'\cdot s)\]
for all $s\in S$ and all $r,r'\in R.$
\end{defi}

\begin{defi}
Let $G$ be a group and $X$ be a $G$-set. We say that $X$ together with a map $p:X\to G$ 
is an {\bf augmented rack} when
it satisfies the augmentation identity, i.e. 
\[p(x\cdot g)\,=\,g^{-1}\,p(x)\, g\]
for all $g\in G$ and all $x\in X.$
\end{defi}

We observe that for any augmented rack $p:X\to G$, one may define a rack operation 
on $X$ as $x\lhd x'\,:=\,x\cdot p(x')$ for all $x,x'\in X$.  Then, the map $p$ 
becomes an equivariant morphism of racks (with respect to the given $G$-action on $X$ 
and the conjugation action on 
the group $G$). Augmented racks are in fact the Yetter-Drinfel'd modules over 
the Hopf algebra $G$ (in the symmetric monoidal category of sets), or in other words,
the Drinfel'd center of the symmetric
monoidal category of $G$-modules, see \cite{FreYet}.   

\begin{exa}
There are many examples of augmented racks. For example, for each rack $R$, the canonical morphism 
$R\to{\rm Aut}(R)$ and the morphism $i:R\to{\rm As}(R)$ are augmented racks.
\end{exa}

We now introduce the notion of a crossed module of racks:

\begin{defi}   \label{definition_cr_mod}
A {\bf crossed module of racks} is a morphism of racks $\mu:R\to S$ together with an action of $S$ on
$R$ by automorphisms such that:
\begin{enumerate}
\item $\mu$ is equivariant, i.e. $\mu(r\cdot s)\,=\,\mu(r)\lhd s$ for all $s\in S$, $r\in R$ and
\item Peiffer's identity is satisfied, i.e. $r\cdot \mu(r')\,=\,r\lhd r'$ for all $r,r'\in R$. 
\end{enumerate}
\end{defi}

\subsection{Examples of crossed modules of racks}

We now provide some important classes of examples.


\begin{exa}  \label{embedding}
Let $\mu:M\to N$ be a crossed module of groups. Passing to the associated conjugation racks of $M$ and
$N$, we obtain a crossed module of racks. Indeed, the group morphism $N\to{\rm Bij}(M)$ gives rise to
a morphism of racks, and thus we have a rack action of the conjugation rack $N$ on the conjugation rack 
$M$. Moreover we have:
\[(r\lhd r')\cdot n\,=\,((r')^{-1}rr')\cdot n\,=\,(r'\cdot n)^{-1}(r\cdot n)(r'\cdot n)\,=\,
(r\cdot n)\lhd(r'\cdot n),\]
for all $n\in N$ and all $m,m'\in M$, 
where we have used the fact that the group $N$ acts on $M$ by automorphisms. Finally, the 
equivariance 
condition for $\mu$ and the Peiffer identity follow from the analogous conditions for the 
crossed module of groups. \end{exa}

\begin{rem}
Let us recall here a mechanism to construct {\it explicit} examples of 
crossed modules of groups, and thus, {\it a fortiori}, of racks. 
In fact, one can construct crossed modules in an explicit way from cohomology classes 
$[\theta]\in H^3(G,V)$ for some group
$G$ and a $G$-module $V$, cf \cite{Wag}. 
Indeed, choose an injective presentation of $V$, i.e. a short exact sequence
\begin{equation}    \label{short_exact}
0\to V\to I\to Q\to 0,
\end{equation}
where $I$ is an injective $G$-module. 
The long exact sequence in cohomology contains the connecting map
\[\partial:H^2(G,Q)\to H^3(G,V),\]
which by injectivity of $I$ is an isomorphism. 
There exists thus a unique class $[\alpha]\in H^2(G,Q)$ with 
$\partial[\alpha]=\theta$. To $[\alpha]$, one may associate an abelian extension
\[0\to Q\to Q\times_{\alpha}G \to G\to 0,\]
and this extension can be spliced together with the short exact sequence 
(\ref{short_exact}) to give a crossed module
\[0\to V\to I \to Q\times_{\alpha}G \to G\to 0.\]
Under the isomorphism between equivalence classes of crossed modules of  
groups with kernel $V$ and cokernel $G$ and
$H^3(G,V)$, this crossed module corresponds to $[\theta]$. 
In many cases, this crossed module is explicitely constructible,
see \cite{Wag}. Observe that for this construction, it suffices 
to have $\partial$ surjective, and this only in degree two, 
so one does not need an injective presentation.
\end{rem}    
   
\begin{exa}
An augmented rack $p:X\to G$ is an example of a crossed module of racks. Indeed, we have already 
remarked that $X$ may be equipped with a rack operation making $p$ an equivariant morphism of racks. 
Then, $G$ acts on the rack $X$, because the group morphism $G\to{\rm Bij}(X)$ is also a morphism
of conjugation racks (cf Lemma \ref{rack_action}). Moreover, $G$ acts by automorphisms, because
\begin{eqnarray*}
(x\lhd x')\cdot g&=&(x\cdot p(x'))\cdot g\\
&=& x\cdot (g g^{-1} p(x') g) \\
&=& (x\cdot g)\cdot p(x'\cdot g) \\
&=& (x\cdot g)\lhd (x'\cdot g)
\end{eqnarray*}
for all $g\in G$ and all $x,x'\in X.$
We have already mentioned that the equivariance of $p$ comes from the augmentation identity.
The Peiffer identity comes from the definition of the rack operation on $X$. 
\end{exa} 

\begin{exa}  \label{gen_augm_rack}
There is also a {\it generalized augmented rack}, i.e. an augmented rack of racks instead of 
groups. For this, let $R$ be a rack and $X$ be an $R$-module (in the sense of Definition   
\ref{definition_rack_action}). Suppose there is a map $p:X\to R$ which satisfies the \emph{generalized
augmentation identity}, i.e. 
\[p(x\cdot r)\,=\,p(x)\lhd r\]
for all $r\in R$ and all $x\in X$.
Then this generalized augmented rack defines a crossed module of racks. Namely, $X$ becomes a rack
with the product 
\[x\lhd y\,:=\,x\cdot p(y),\]
for all $x,y\in X$. For this, note first that $p$ is a morphism of racks:
\[p(x\lhd y)\,=\,p(x\cdot p(y))\,=\,p(x)\lhd p(y).\]
We verify the rack identity using the fact that $\cdot$ is a rack action and the generalized augmentation 
identity:
\begin{eqnarray*}
(x\lhd y)\lhd z&=&(x\cdot p(y))\cdot p(z) \\
&=& (x\cdot p(z))\cdot (p(y)\lhd p(z)) \\
&=& (x\lhd z)\cdot p(y\lhd z)  \\
&=& (x\lhd z)\lhd (y\lhd z).
\end{eqnarray*}
Furthermore, $R$ acts by automorphisms on the rack $X$:
\begin{eqnarray*}
(x\lhd y)\cdot r &=& (x\cdot p(y))\cdot r \\
&=& (x\cdot r)\cdot(p(y)\lhd r)) \\
&=& (x\cdot r)\cdot p(y\cdot r) \\
&=& (x\cdot r)\lhd (y\cdot r). 
\end{eqnarray*} 
It is clear that $p$ is equivariant by the generalized augmentation identity, and that the Peiffer 
identity follows from the definition of the rack product on $X$. Thus in conclusion, 
$p:X\to R$ is a crossed module of racks. 
\end{exa} 

These three classes of examples are ordered here with growing generality, i.e. crossed modules of 
groups are particular augmented racks, and augmented racks are particular generalized augmented racks.

\subsection{Relations between these classes of examples}   \label{relation_examples}

It turns out that the last class of examples from the previous section 
is equivalent to crossed modules of racks:

\begin{prop}  \label{cr_mod_vs_gen_augm_rack}
There is a one-to-one correspondence between crossed modules of racks and generalized augmented
racks. 
\end{prop}

\begin{proof} In Example \ref{gen_augm_rack}, we gave the construction of a crossed module of racks from a 
generalized augmented rack. Conversely, given a crossed module of racks $\mu:R\to S$, forgetting the 
rack structure on $R$ leaves us with a generalized augmented rack. Clearly, the two constructions
are inverse to each other. \end{proof}

Thanks to this proposition, we will very often regard crossed modules simply as a rack $R$, an $R$-module
$X$ and an equivariant map $p:X\to R$. 

\begin{prop}  \label{cr_mod_racks_to_cr_mod_grps}
The functor ${\rm As}:{\tt Racks}\to{\tt Grp}$ sends crossed modules of racks to crossed modules 
of groups. 
\end{prop}

\begin{proof} Let $p:R\to S$ be a crossed module of racks. Then by functoriality, 
$p$ extends to a group homomorphism
$p:{\rm As}(R)\to{\rm As}(S)$, which is just $p$ on elements of word length one.

The rack action $R\times S\to R$ gives rise to a rack
morphism $S\to{\rm Bij}(R)$ that extends by the universal property to a group homomorphism
${\rm As}(S)\to{\rm Bij}(R)$. The action of ${\rm As}(S)$ on $R$ is then extended to ${\rm As}(R)$
demanding that it should be an action by group automorphisms:
\[(r_1r_2)\cdot s\,=\,(r_1\cdot s)(r_2\cdot s),\]
for all $r_1,r_2\in{\rm As}(R)$ and all $s\in{\rm As}(S)$. This equation first extends the action 
of ${\rm As}(S)$
to an action on the free group $F(R)$, but as $(r_1\lhd r_2)\cdot s=(r_1\cdot s)\lhd(r_2\cdot s)$, 
this passes to the quotient ${\rm As}(R)$ of $F(R)$. 

The map $p$ is equivariant, i.e. $p(r\cdot s)=p(r)\lhd s=s^{-1}p(r)s$ for all $r\in {\rm As}(R)$ and all 
$s\in {\rm As}(S)$, because this statement holds for elements of word length one, and extends to all elements
$r\in {\rm As}(R)$ by the fact that $p$ is a group homomorphism and the action is an action by automorphisms.
It extends finally 
to all $s\in {\rm As}(S)$ using the action property. 

The Peiffer identity is shown in a similar way.\end{proof}

\begin{rem}
One could investigate whether the functors ${\rm Aut}:{\tt Racks}\to{\tt Grp}$ and 
${\rm Bij}:{\tt Racks}\to{\tt Grp}$ also send crossed modules of racks to crossed modules of groups.
\end{rem}                

One can do the replacement of racks by the associated groups also partially, i.e. replace in a 
crossed module $p:X\to R$ the rack $R$ by the group ${\rm As}(R)$. Note that the statements of the following two propositions compose to give back the statement of the previous proposition.  

\begin{prop}   \label{cr_mod_racks_to_augmented_racks}
Given a crossed module of racks, the corresponding map $p:X\to{\rm As}(R)$ is an augmented rack. 
\end{prop}

\begin{proof} The rack homomorphism $p:X\to R$ extends by the universal property to a map $p:X\to{\rm As}(R)$.
The action of $R$ on $X$ (given by a rack homomorphism $R\to{\rm Bij}(X)$) extends by the universal
property to a group action of ${\rm As}(R)$ on $X$. The identity
\[p(x\cdot r)\,=\,r^{-1}p(x)r\]
is true for elements $r$ of word length one in ${\rm As}(R)$ (and all elements $x\in X$), 
because of $p(x\cdot r)=p(x)\lhd r$. 
This then extends to all elements using the action property. \end{proof}

In the same way, one can replace the $G$-set $X$ in an augmented rack, regarded as a rack, by the
associated group ${\rm As}(X)$:

\begin{prop}    \label{augm_racks_to_cr_mod_grps}
Given an augmented rack $p:X\to G$, the induced map $p:{\rm As}(X)\to G$ is a crossed module of groups.
\end{prop} 

\begin{proof} Regarding $X$ as a rack, the rack morphism $p:X\to {\rm Conj}(G)$ gives rise by the universal 
property to a group homomorphism $p:{\rm As}(X)\to G$. Then $G$ acts on the subrack $X\subset{\rm As}(X)$
by automorphisms, and imposing
\[(x_1x_2)\cdot g\,=\,(x_1\cdot g)(x_2\cdot g)\]
for all $x_1,x_2\in X$ and all $g\in G$ extends the action on $X$ to an action on the free group 
$F(X)$.  Then, using  $(x_1\lhd x_2)\cdot g=(x_1\cdot g)\lhd(x_2\cdot g)$, it extends to an action 
by automorphisms on the quotient 
${\rm As}(X)$ of $F(X)$. 

One shows the equivariance of $p$ and the Peiffer identity as in the proof of
Proposition \ref{cr_mod_racks_to_cr_mod_grps}. \end{proof}

\begin{rem}
Proposition \ref{augm_racks_to_cr_mod_grps} is due to Fenn-Rourke \cite{FenRou} p.356.
\end{rem}

\begin{rem}
The proofs of Propositions \ref{cr_mod_racks_to_cr_mod_grps} and \ref{augm_racks_to_cr_mod_grps} are
very similar, but observe the difference in their the statements. 
\end{rem}

We therefore have the three main classes of examples:
\[{\rm cr-mod}({\tt Grp})\,\subset\,{\tt augm\,\,Racks}\,\subset\,{\rm cr-mod}({\tt Racks})\]
Furthermore, we have the pair of adjoint functors 
\[\xymatrix{{\rm Conj} : {\tt Grp}    \ar@<4pt>[r]
& \ar[l]   {\tt Racks} : {\rm As}  } \]
which extend to functors going back and forth between these classes.  
The (elementary) 
composition functors which arise have components of the form
\[{\rm Conj}({\rm As}(R))\]
for a rack $R$, or
\[{\rm As}({\rm Conj}(G))\]
for a group $G$. In general ${\rm As}({\rm Conj}(G))$ is far from being $G$
and ${\rm Conj}({\rm As}(R))$ is far from being $R$. 
This is the information loss one suffers by going from crossed modules of racks to augmented racks, or 
from augmented racks to crossed modules of groups.      

\subsection{Algebraic structure on the category ${\rm cr-mod}({\tt Racks})$}

It is well-known (see for example \cite{Kas} p. 319) that the category of 
augmented racks over a fixed group carries a braiding:

\begin{lem} 
For two augmented racks $p_i:X_i\to G$, $i=1,2$, with respect to a fixed group $G$, their tensor product 
$X\otimes Y$ is defined as
$X\times Y$ with the action \[(x,y)\cdot g\,=\,(x\cdot g,y\cdot g)\] and the equivariant map $p:X\times Y
\to G$ defined by $p(x,y)=p_1(x)p_2(y)$. Then the formula
\[c_{X,Y}:X\otimes Y\to Y\otimes X,\,\,\,\,\,\,c_{X,Y}(x,y)\,:=\,(y,x\cdot p(y))\]
defines a braiding on the category of augmented racks over the fixed group $G$. 
\end{lem}

\begin{rem}
For the larger category of crossed modules of racks over a fixed rack $R$, one can use Proposition
\ref{cr_mod_racks_to_augmented_racks} to define a braiding on the category of crossed modules of 
racks over a fixed rack. 
\end{rem} 

\section{Crossed modules and categorical racks} \label{second}

\subsection{From categories to crossed modules}

It is well known that crossed modules of groups $\mu:M\to N$ are in one-to-one correspondence with 
category objects in the category of groups ${\tt Grp}$, also known as strict $2$-groups or 
categorical groups \cite{BL, FB}. 
We recall the correspondence given in \cite{BL}: Given $\mu:M\to N$, we construct a strict $2$-group by taking 
$G_0:=N$ as the group of objects and the semi-direct product $G_1:=M\rtimes N$ as the group of morphisms. 
In the other direction, given a strict $2$-group $\xymatrix{G_1\ar@<2pt>[r]^{s}
\ar@<-2pt>[r]_{t} & G_0}$ we take $N:=G_0$
and $M:=\ker(s)$, where $s$ is the source map. The map $\mu$ is given by the restriction of the target
map $t$ to $\ker(s)$. We will call the passage to $t:\ker(s)\to G_0$ the {\it standard construction}.
From this well-known construction, we note following lemma:

\begin{lem}   \label{semi-direct-product-construction}
Let $M$ and $N$ be groups such that $N$ acts on $M$ by automorphisms, and let $\mu:M\to N$ be an
equivariant homomorphism. Then the semi-direct product $M\rtimes N$ carries a unique structure of a category
such that the objects are $N$, the morphisms $M\rtimes N$, the source $(m,n)\mapsto n$, the target
$(m,n)\mapsto \mu(m)n$ and the identity $n\mapsto (1_M,n)$. 
\end{lem} 

In fact, there are three more objects which are equivalent to strict $2$-groups, see \cite{Lod}:

\begin{enumerate}
\item groups $G$ with a subgroup $N$ and two homomorphisms $s,t:G\to N$ with $s|_N=\id_N$,
$t|_N=\id_N$ and $[\ker(s),\ker(t)]=1$ (these are called {\it 1-cat groups}),
\item simplicial groups with Moore complex of length one,
\item group objects in the category of (small) categories.
\end{enumerate}  

We desire to have correspondences of a similar type for crossed modules of racks.

\begin{defi}
A {\bf 1-cat rack} consists of a pointed rack $R$, a subrack $N$ and two rack morphisms $s,t:R\to N$ such that
$s|_N=\id_N$, $t|_N=\id_N$ and $\ker(s)$ and $\ker(t)$ act trivially on each other.
\end{defi}

One approach to the construction of a similar correspondence to that described above in the group case
is to consider category objects in the category of racks, see \cite{CCES}. 

\begin{defi}
A {\bf strict $2$-rack} or {\bf categorical rack} is a category object in the category of racks.  
That is, a strict $2$-rack consists of 
two pointed racks $R_0$ (rack of objects)  and $R_1$ (rack of morphisms) equipped with rack morphisms 
$s,t:R_1\to R_0$ (source and 
target), $i:R_0\to R_1$ (identity-assigning) and $\circ:R_1\times_{R_0}R_1\to R_1$ 
(composition) such that the usual axioms 
of a category are satisfied. 
\end{defi}

One crucial property of the composition in a strict $2$-group is the {\it middle four exchange 
property} which simply means that the composition $\circ$ is a morphism of groups:
\[(g_1g_2)\circ(f_1f_2)\,=\,(g_1\circ f_1)(g_2\circ f_2).\]
We observe that this property holds for all morphisms $f_1:a\mapsto b$, $f_2:a'\mapsto b'$, $g_1:b\mapsto c$,
$g_2:b'\mapsto c'$, where the constraints on the domains and ranges reflect the composability. 

From this property, one deduces that the kernel of the source map, $\ker(s)$, and the kernel of the 
target map, $\ker(t)$, commute and that the group product on $G_1$ uniquely determines the composition.
Note that as the inclusion of identities the map $G_0\hookrightarrow G_1$ is a group homomorphism and the 
identity $1\in G_1$ is both the identity with respect to the composition and unit with respect 
to the group product. 
We see that we have a similar property in the framework of strict $2$-racks:

\begin{prop}  \label{MFE}
The middle four exchange property for a strict $2$-rack implies that $\ker(s)$ and $\ker(t)$  
act trivially on each other. 
\end{prop}

\begin{proof} For strict $2$-racks, the middle four exchange property states
\[(g_1\lhd g_2)\circ(f_1\lhd f_2)\,=\,(g_1\circ f_1)\lhd (g_2\circ f_2)\]
for all $f_1:a\mapsto b$, $f_2:a'\mapsto b'$, $g_1:b\mapsto c$, and
$g_2:b'\mapsto c'$. By choosing $b=c=1$, $g_1=1$ and by using the fact that the identity in 
$1\in R_1$ is both unit and identity we deduce that
\[f_1\lhd f_2\,=\,f_1\lhd (g_2\circ f_2)\]
for all $f_1\in\ker(t)$. 
 
In the special case when $a'=b'=1$ with $f_2=1$ and thus $g_2:1\mapsto c$, i.e.
$g_2\in\ker(s)$ we obtain from the above that
\[f_1\,=\,f_1\lhd g_2\]
for all $f_1\in\ker(t)$ and all $g_2\in\ker(s)$. 
This means that elements from $\ker(s)$ act trivially on elements of $\ker(t)$.

In the same way, choosing $a=b=1$ and $b'=c'=1$, $f_1=1$ and $g_2=1$, and therefore 
$f_2\in\ker(t)$ and $g_1\in\ker(s)$, we obtain
\[g_1\,=\,g_1\lhd f_2,\]
which means that $\ker(t)$ acts also trivially on $\ker(s)$.\end{proof} 
 
\begin{cor}
A strict $2$-rack has an underlying 1-cat rack, i.e. a pointed rack $R$ together with a subrack
$N$ and two homomorphisms $s,t:R\to N$ such that $s|_N=\id_N$, $t|_N=\id_N$ and $\ker(s)$ and $\ker(t)$
act trivially on each other.
\end{cor}

\begin{proof} Indeed, define $R:=R_1$, $N:=i(R_0)$, and take source and target maps $s,t:R\to N$. 
The properties $s|_N=\id_N$ and $t|_N=\id_N$ come from $t\circ i=\id_{R_0}$ and $s\circ i=\id_{R_0}$.
By Proposition \ref{MFE}, we have that $\ker(s)$ and $\ker(t)$
act trivially on each other.\end{proof}

\begin{rem}
This correspondence from strict $2$-racks to 1-cat racks is actually functorial. 
We leave it to the interested reader to define the necessary (2-)category structure 
on both classes of objects in order to make this a mathematical statement.
Similar remarks apply to the propositions in this and the next subsection.
\end{rem}   

\begin{rem}
Unfortunately, the four equations following from the middle four exchange property in a strict $2$-rack:
\begin{enumerate}
\item $g_1\circ(f_1\lhd f_2)\,=\,(g_1\circ f_1)\lhd f_2$ with the restriction $f_2\in\ker(t)$,
\item $f_1\lhd f_2\,=\,f_1\lhd(g_2\circ f_2)$ with the restriction $f_1\in\ker(t)$,
\item $g_1\lhd g_2\,=\,g_1\lhd(g_2\circ f_1)$ with the restriction $g_1\in\ker(s)$, and
\item $(g_1\lhd g_2)\circ f_1\,=\,(g_1\circ f_1)\lhd g_2$ with the restriction $g_2\in\ker(s)$,
\end{enumerate}
do not seem to enable us to reconstruct the composition starting from the rack product, 
or vice-versa (as we can in the case for 
strict $2$-groups). We leave this observation as a question for future study.  
\end{rem}   

\begin{prop}
A strict $2$-rack gives rise, via the standard construction, to a crossed module of racks.
\end{prop}

\begin{proof} Recall from Proposition \ref{cr_mod_vs_gen_augm_rack} that a crossed module of 
racks consists of a rack $R$, an $R$-module $X$ and an equivariant map
$p:X\to R$. Given a strict $2$-rack $(R_0, R_1, s, t, i, \circ)$, we
define $R:=R_0$. The rack $R$ acts on $X\,:=\,\ker(s)$ by
\[x\cdot r\,:=\,x\lhd i(r).\]
Indeed, $\ker(s)$ is preserved by this action since
\[s(x\cdot r)\,=\,s(x\lhd i(r))\,=\,s(x)\lhd s(i(r))\,=\,1\lhd r\,=\,1.\]  The fact that 
this is an action follows from the rack identity:
\begin{eqnarray*}
(x\cdot r)\cdot r'&=&(x\lhd i(r))\lhd i(r')  \\
&=& (x\lhd i(r'))\lhd(i(r)\lhd i(r'))  \\
&=& (x\cdot r')\cdot (r\lhd r').
\end{eqnarray*} 
Then the map $t|_{\ker(s)}$ is equivariant, because for all $x\in X$ and all $r\in R$
\[t(x\cdot r)\,=\,t(x\lhd i(r))\,=\,t(x)\lhd t(i(r))\,=\,t(x)\lhd r.\] 
\end{proof} 

\begin{rem}
It is a natural question to ask whether the generalized augmented rack 
$t|_{\ker(s)} : \ker(s) \to R$ constructed in the proof {\it is} actually a crossed module of racks  
in the sense of Definition \ref{definition_cr_mod} rather than using Proposition 
\ref{cr_mod_vs_gen_augm_rack} to {\it make it} a crossed module of racks, 
because $\ker(s)$ already carries a rack structure as a subrack of $R_1$.

In fact, it is clear that $t|_{\ker(s)}$ is a 
morphism of racks, and it is easy to see that the above action of $R$ on $X$ is by automorphisms.
The only thing which is not clear is Peiffer's identity. 

This means that starting from a strict $2$-rack we can always {\bf define} a rack product on
$X=\ker(s)$ such that $t:X\to R$ becomes a crossed module, but there is, a priori, no relation to
the {\bf induced} rack product from $R_1$.
\end{rem}

\subsection{From crossed modules to categories}

We will now indicate how one may try to perform the reverse direction, i.e. construct strict $2$-racks
from crossed modules of racks or from 1-cat racks. For this, we use the fact that we are able to pass from 
crossed modules of racks to augmented racks (or directly to crossed modules of groups)
and from augmented racks to crossed modules of groups, see Section \ref{relation_examples}.

Indeed, what we are lacking is an analogue of Lemma \ref{semi-direct-product-construction} in the 
pure framework of racks. The most natural approach would be to use 
the analogue of the semi-direct product in the realm of racks, i.e. 
the hemi-semi-direct product, see Definition
\ref{definition_hs_product}. Actually, this does not work.  
We were unable to combine a rack $R$, an $R$-module $X$ and an equivariant
map $p:X\to R$ into a rack structure which gives even a pre-category in the category of racks.
However, this can be done in some special cases, such as when $p$ has trivial image $\{1\}$, or if
the rack product is trivial on $R$, or if the rack action of $p(X)$ on $R$ is trivial.

We are, nevertheless, able to do the following: Given an augmented rack or a crossed module of racks, 
one can use the functor ${\rm As}$ as described 
in Section \ref{relation_examples} to associate to it a crossed module of groups. To this, one can apply
the usual cronstruction to obtain a strict $2$-group. Finally, one may use the following proposition
from \cite{CCES}:

\begin{prop}[Carter-Crans-Elhamdadi-Saito \cite{CCES}]
The functor ${\rm Conj}:{\tt Grp}\to{\tt Racks}$ sends strict $2$-groups to strict $2$-racks.
\end{prop}

As in Section \ref{relation_examples}, one can obtain crossed modules of groups by starting with a 
crossed module of racks (applying the functor ${\rm As}$ on both racks) or starting with an augmented
rack (applying ${\rm As}$ only on the $G$-set). Another way would be to regard the augmented rack as
a crossed module of racks - this would result in applying ${\rm As}$ also on the group $G$.  

Yet another option, this time {\it without} using the functor ${\rm As}$, requires more structure: 

\begin{prop}   \label{category_from_augm_rack}
Given an augmented rack $p:X\to G$ such that $X$ is a $G$-module, i.e. an abelian group with a 
linear $G$-action, and $p:X\to G$ is a homomorphism, then the usual semi-direct product construction
from Lemma \ref{semi-direct-product-construction} gives rise to a strict $2$-group.
\end{prop}

\begin{rem}
One may ask what happens when we start with a crossed module of racks (or an augmented 
rack), associate an augmented rack (or a crossed module of groups) to it, perform the $2$-group construction,
regard it as a strict $2$-rack, 
and then reconstruct a crossed module of racks from it. 

Given a crossed module of racks $\mu:R\to S$, we associate to it the crossed module of groups
$\mu:{\rm As}(R)\to{\rm As}(S)$, which can then be regarded as a strict $2$-group in the usual way.   
Then the standard construction of a crossed module of groups from a strict $2$-group applies here, and gives 
$t|_{\ker(s)}:\ker(s)\to {\rm As}(S)$. It is clear that $\ker(s)={\rm As}(R)$ (as the map $s$ is the 
projection from the semi-direct product ${\rm As}(R)\rtimes{\rm As}(S)$ to ${\rm As}(S)$, see
Lemma \ref{semi-direct-product-construction}). Thus we get back the 
crossed module of groups between the associated groups. In case we started with a crossed module of 
racks $\mu:R\to S$ such that the augmented rack $\mu:R\to{\rm As}(S)$ satisfies 
the hypotheses of Proposition \ref{category_from_augm_rack}, then the crossed module we get from the 
above construction is $\mu:R\to{\rm As}(S)$. 

Observe that on the augmented rack $p:X\to G$ from Proposition \ref{category_from_augm_rack},
there are two rack structures on $X$. One is the trivial rack structure coming from the conjugation
rack with respect to the abelian group structure on $X$, and the other comes from the 
$(x,y)\mapsto x\cdot p(y)$-construction.  
\end{rem}

\section{Crossed modules of racks and trunks} \label{third}

Our main idea to go beyond the constructions from the previous section is to associate to a 
crossed module of racks $\mu:R\to S$ not a category, but a {\it trunk}, see \cite{FRS}.

\subsection{Basic definitions}  

\begin{defi}
A {\bf trunk} is a directed graph $\Gamma$ together with a collection of oriented squares 
\[\xymatrix{ C  \ar[r]^c & D \\
             A \ar[u]^b \ar[r]^a & B \ar[u]_d }\]
called \emph{preferred squares}. 
\end{defi}

In categorical language, the set of objects consists of the vertices of $\Gamma$ and 
the set of morphisms consists of the edges of $\Gamma$.  There are then source and 
target maps for arrow/edge. We notice that we are missing the identity-assigning and 
composition morphisms. Instead of these, however, we have 
preferred squares (commutative square diagrams) which, in some sense, replace the composition.

We now introduce the missing identities:

\begin{defi}
A {\bf pointed trunk} is a trunk equipped with a chosen edge $e_A:A\to A$ for each 
vertex $A$ and the following preferred squares for any given edge $a:A\to B$:
\[\xymatrix{ A  \ar[r]^a & B \\
             A \ar[u]^{e_A} \ar[r]^a & B \ar[u]_{e_B} }\]
The edges $e_A:A\to A$ are called {\bf identities}. 
\end{defi}

We will assume that all our racks are pointed and all our trunks have identities. 
Now, in order to model racks in terms of trunks, we pass to the so-called {\it corner trunks}, see
\cite{FRS} p. 324:

\begin{defi}
A {\bf corner trunk} is a trunk which satisfies the two corner axioms:
\begin{enumerate}
\item[(C1)] Given edges $a:A\to B$ and $b:A\to C$, there are unique edges $a\lhd b:C\to D$ and 
$a\rhd b: B\to D$ such that the following square is preferred
\[\xymatrix{ C  \ar[r]^{a\lhd b} & D \\
             A \ar[u]^{b} \ar[r]^a & B \ar[u]_{a\rhd b} }\]
\item[(C2)] In the following diagram, if the squares $(ABCD)$, $(BDYT)$ and $(CDZT)$ are preferred, then 
then the diagram can be completed, as shown by the dotted lines, such that the squares $(ABXY)$, $(ACXZ)$ 
and $(XYZT)$ are 
preferred. 
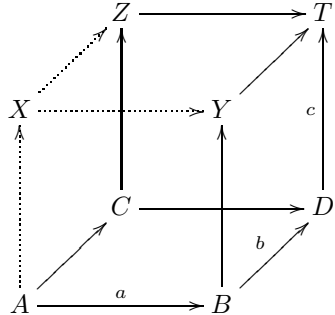
\[\xymatrix{   & Z \ar[rr] &  &  T \\
             X \ar@{..>}[ur] \ar@{..>}[rr] & & Y \ar[ur] & \\
               & C \ar[uu] \ar[rr] & & D \ar[uu]^c \\
             A \ar@{..>}[uu] \ar[ur] \ar[rr]^a & & B \ar[uu] \ar[ur]^b &  }\]
\end{enumerate}
\end{defi}

\begin{lem} \label{cornerbinary}
In a corner trunk, the binary operations $\rhd$ and $\lhd$ satisfy: 
\begin{enumerate}
\item $(a\lhd b)\lhd(b\rhd c)\,=\,(a\lhd c)\lhd(b\lhd c)$
\item $(b\rhd c)\rhd(b\rhd a)\,=\,(b\lhd c)\rhd(c\rhd a)$
\item $(b\rhd a)\lhd(b\rhd c)\,=\,(b\lhd c)\rhd(a\lhd c)$
for all arrows $a,b,c$.
\end{enumerate}
\end{lem}

In fact, a set $X$ with two operations $\rhd$ and $\lhd$ (which are bijective 
and) satisfy the above three identities, is called a {\it birack}. Biracks 
also serve to construct link invariants, see e.g. \cite{CEGN}. 

The case which is most interesting to us is when one of the two 
operations $\rhd$ or $\lhd$ is trivial. 
Then, by Lemma \ref{cornerbinary} above, the other operation 
satisfies a (left or right) rack identity. This means that these types
of corner trunks codify racks. The notion of a corner trunk can be pointed in an obvious way. 

\begin{exa}
Any rack $(R,\lhd)$ gives rise to a corner trunk ${\mathcal T}(R)$, 
called the {\it rack trunk}, see \cite{FRS} p. 327.
The trunk ${\mathcal T}(R)$ consists of a single vertex $*$, while the preferred squares are given by the 
rack operation
\[\xymatrix{ {*}  \ar[r]^{a\lhd b} & {*} \\
             {*} \ar[u]^{b} \ar[r]^a & {*.} \ar[u]_{b}  }   \]
\end{exa}

\begin{exa}  \label{action_trunk}
{\it The action rack trunk}, cf \cite{FRS} p. 329. Let $X$ be an $R$-set where $R$ is a rack. 
From this data, we construct a trunk ${\mathcal T}_X(R)$. Namely, we take $X$ as the set of vertices 
and edges of the form $x\stackrel{r}{\to}x\cdot r$ for $r\in R$ and $x\in X$. The preferred squares
are then of the form:
\[\xymatrix{ x\cdot r'  \ar[r]^{r\lhd r'} & (x\cdot r)\cdot r' \\
             x  \ar[u]^{r'} \ar[r]^r & x\cdot r \ar[u]^{r'} }\]
for all $r,r'\in R$ and $x\in X$.
Observe that this fits together in the upper right hand corner because for our right action we have
\[(x\cdot r)\cdot r'\,=\,(x\cdot r')\cdot (r\lhd r').\]
As in \cite{FRS}, we see that ${\mathcal T}_X(R)$ is indeed a corner trunk. From the categorical point 
of view, we will denoted morphisms $x\stackrel{r}{\to}x\cdot r$ as pairs $(x,r)$, and then we have
source and target maps $\xymatrix{X\times R\ar@<2pt>[r]^{s}\ar@<-2pt>[r]_{t} & X}$ given by
$s(x,r)=x$ and $t(x,r)=x\cdot r$. As already remarked in {\it loc. cit.}, the operation expressed 
by the preferred squares can be expressed as:
\[(x,r)\lhd(x,r')\,=\,(x\cdot r',r\lhd r'),\,\,\,\,\,(x,r')\rhd(x, r)\,=\,(x\cdot r,r').\]
The first formula ressembles the hemi-semi-direct product, see Definition
\ref{definition_hs_product}, but with a composability condition 
(the first components have to be equal to $x$). 

In the special case where a rack $(R,\lhd)$ acts on itself by $x\mapsto x\lhd y$, the action rack trunk
 ${\mathcal T}_R(R)$ is called {\it extended rack trunk} in \cite{FRS} on p. 329. 
\end{exa}


\subsection{From crossed modules of racks to trunks}

Given a crossed module of racks $\mu:R\to S$, we have, in particular, an $S$-set $R$, and we can 
thus apply the construction from Example \ref{action_trunk} to obtain a corner trunk. 
This is, in our opinion, the correct ``categorical object" associated to a crossed 
module of racks, as it is a `sort of trunk' in the 
category of racks.  Thus, we introduce:

\begin{defi}
A {\bf trunkified rack} consists of a trunk map between an action rack trunk (for $R$ acting on $X$)
and the extended rack trunk of $R$. 
\end{defi}

\begin{prop}  \label{correspondence_cr_mod_racks_trunks}
There is a one-to-one correspondence between crossed modules of racks and trunkified racks. 
\end{prop}   
 
\begin{proof}
We first observe that by Proposition \ref{cr_mod_vs_gen_augm_rack}, it is enough to work
with generalized augmented racks instead of crossed modules of racks. A generalized augmented rack
consists of a rack $R$, an $R$-module $X$ and an equivariant map $p:X\to R$.
We can associate to these the extended rack trunk 
${\mathcal T}_R(R)$, the action trunk ${\mathcal T}_X(R)$ and the induced trunk map 
$p:{\mathcal T}_X(R)\to{\mathcal T}_R(R)$ (or $p:{\mathcal T}_X(R)\to{\mathcal T}(R)$). 

In the other direction, we recover from a trunk map ${\mathcal T}_X(R)\to{\mathcal T}_R(R)$
the rack $R$ and the rack action of $R$ on $X$. The trunk map gives an equivariant map, and 
therefore we recover our crossed module. \end{proof}    

\begin{rem}
One can, of course, introduce alternative versions of trunkified racks. For example, one 
could choose to define a
trunkified rack as a trunk in the category of racks. 

In order to associate to a crossed module of racks such a trunkified rack, one idea would be to take
the image of the above trunk map appearing in the proof of Proposition 
\ref{correspondence_cr_mod_racks_trunks}. Unfortunately, we were unable to show that this 
gives a trunk in the category of racks, and so this remains a question for further investigation.
\end{rem}

We conclude this section with the following scheme which may enable us to eventually associate a strict 
$2$-rack
to a crossed module of racks. Let $p:X\to R$ be a crossed module of racks. Denote by
${\rm cat}:{\tt Trunks}\to{\tt Cats}$ the functor from the category of (small) trunks to the 
category of (small) categories that associates to a trunk
${\mathcal T}$ the category whose objects consist of the same set of vertices/objects as the 
trunk, but whose set of morphisms 
(between two fixed objects) are generated by the set of arrows of ${\mathcal T}$ (between these objects)
such that the preferred squares become commutative diagrams.   

We then form the trunk map $p:{\mathcal T}_X(R)\to{\mathcal T}_R(R)$ and take the image trunk 
${\rm im}(p)$ and show that this is a trunk in the category of racks.  We then demonstrate 
that, in general, the functor ${\rm cat}:{\tt Trunks}\to{\tt Cats}$ sends trunks in the 
category of racks to categories in the category of racks (i.e. that ${\rm Cat}$ can be enriched in 
racks).  Finally, the image ${{\rm cat}(\rm im}(p))$ is then the categorical rack associated to $p:X\to R$.


\section{Applications} \label{apps}

We now show that some of these categorical objects that we have associated to crossed modules of 
racks have topological/geometrical applications. 

\subsection{The rack space of a crossed module of racks}

In the paper \cite{FRS}, the authors associate to a rack $X$ a rack space $BX$ by taking the geometric
realization of the cubical nerve $NX$ of the trunk ${\mathcal T}(X)$ associated to $X$. 
They also show that for an
action rack trunk ${\mathcal T}_Y(X)$ (with the rack $X$ acting on $Y$), 
the canonical map $Y\to\{*\}$ induces a
trunk map ${\mathcal T}_Y(X)\to{\mathcal T}(X)$, which gives rise to a covering 
$B_YX\to BX$, see Theorem 3.7 in \cite{FRS}. 

Now starting with a crossed module of racks $p:X\to R$, we have
first of all a trunk map
${\mathcal T}_X(R)\to{\mathcal T}(R)$
inducing the covering $B_XR\to BR$. Then we also have a trunk map
\[p:{\mathcal T}_X(R)\to{\mathcal T}_R(R),\]
where ${\mathcal T}_R(R)$ is the extended rack trunk. 
This map also induces a map of $\Box$-sets between the cubical nerves
\[p:N_X(R)\to N_R(R),\]
and finally a map between the
corresponding rack spaces $B_X(R)\to B_R(R)$, see \cite{FRS} p. 331. 

Both $\Box$-sets $N_X(R)$ and $N_R(R)$ are in fact $\Box$-coverings of $NR$. We will show that $p$ 
is a covering:

\begin{prop}   \label{cr_mod_racks_to_covering} 
Suppose that $B_R(R)$ is arcwise connected and locally arcwise connected. 
Then the geometric realization of the natural map $p:N_X(R)\to N_R(R)$ is a covering 
of topological spaces. 
\end{prop}

\begin{proof} 
The lifting theorem for a continuous map into the base of a covering 
implies that we only need to show that $\pi_1(B_XR,x)\,\subset\,\pi_1(B_RR,p(x))$. 

Now recall the fact that the fundamental group $\pi_1(B_YX,y)$ is just the stabilizer of $y$,
i.e.
\[\pi_1(B_YX,y)\,=\,{\rm Stab}_y\,\subset\,{\rm As}(X),\]
see Proposition 4.5 in \cite{FRS}. In particular, we have $\pi_1(BX,*)\,=\,{\rm As}(X)$. Thus in order 
to show the claim, we just need to show that the subgroups $\pi_1(B_XR,x)$ and $\pi_1(B_RR,p(x))$
satisfy
\[\pi_1(B_XR,x)\,\subset\,\pi_1(B_RR,p(x))\] 
as subgroups of ${\rm As}(R)$. This follows from the inclusion of the corresponding stabilizers
${\rm Stab}_x\,\subset\,{\rm Stab}_{p(x)}$.\end{proof}

In this sense, we can associate to each crossed module of racks a covering of rack spaces. We anticipate 
that this will 
serve to enhance link invariants.

\begin{rem}
It would be interesting to know what kind of local property these rack spaces have. In order
to have a nice theory of covering spaces, one would like to work with topological spaces 
which are, for example, arcwise connected and locally simply connected.
\end{rem}

\subsection{Crossed modules of racks from link coverings}

Here we are doing in some sense the inverse construction with respect to what we did in the previous section.
Namely, given a covering (and a link), we associate to it a crossed module of racks. 

We now consider an inverse construction, in some sense:  Suppose given augmented racks $p_i:X_i\to G_i$ 
for $i=0,1$ and a 
commutative diagram
\[\xymatrix{ X_1  \ar[r]^{p_1} \ar[d]^{\alpha} & G_1 \ar[d]^{\beta}  \\
             X_0 \ar[r]^{p_0} & G_0  }\]
where $\beta$ is a group homomorphism and $\alpha$ is a morphism of group-sets over $\beta$, i.e.
for all $x\in X_1$ and all $g\in G_1$, we have
\[\alpha(x*g)\,=\,\alpha(x)\cdot\beta(g).\]
Here we have written the right actions $x*g$ for the action of $G_1$ on $X_1$, and $y\cdot h$ 
for the action of $G_0$ on $X_0$. We then ask:  Under which conditions is $\alpha:X_1\to X_0$ a 
crossed module of racks ?

\begin{prop}  \label{two_augmented_racks}
Suppose that there is a right action of $G_0$ on $X_1$, denoted $(x,g)\mapsto x\circ g$, such that
\begin{enumerate}
\item $x*g\,=\,x\circ\beta(g)$ for all $g\in G_1$ and all $x\in X_1$,
\item $\alpha$ is equivariant, i.e. 
$\alpha(x\circ g)\,=\,\alpha(x)\cdot g,$ for all $g\in G_0$ and all $x\in X_1$, and
\item $(y*p_1(x))\circ g\,=\,(y\circ g)*p_1(x\circ g)$ for all $g\in G_0$ and all $x,y\in X_1$.
\end{enumerate}
Then $\alpha:X_1\to X_0$ is a crossed module of racks for the augmented racks 
$p_i:X_i\to G_i$ for $i=0,1$.
\end{prop} 

\begin{proof} The sets $X_0$ and $X_1$ become racks via the usual definitions: $x\lhd y\,:=\,x*p_1(y)$ 
for $x,y\in X_1$ and $x\lhd y\,:=\,x\cdot p_0(y)$ for $x,y\in X_0$. 
The map $\alpha$ is then a morphism of racks since
\begin{eqnarray*}
\alpha(x\lhd y)&=&\alpha(x* p_1(y))   \\
&=&\alpha(x)\cdot\beta(p_1(y))          \\
&=&\alpha(x)\cdot p_0(\alpha(y))        \\
&=&\alpha(x)\lhd\alpha(y)
\end{eqnarray*}

Using the map $p_0$, the action $\circ$ of $G_0$ on $X_1$ induces an action of $X_0$ on $X_1$.  

Property {\it 3} clearly translates into the fact that the action $\circ$ is by rack autmorphisms:   
\[(y\lhd x)\circ g\,=\,(y*p_1(x))\circ g\,=\,(y\circ g)*p_1(x\circ g)\,=\,(y\circ g)\lhd(x\circ g).\]

Property {\it 2} is the equivariance of map $\alpha$:
\[\alpha(x\circ p_0(y))\,=\,\alpha(x)\cdot p_0(y)\,=\,\alpha(x)\lhd y,\] and

Peiffer's identity is satisfied thanks to Property {\it 1}
\[x\circ p_0(\alpha(y))\,=\,x\circ \beta(p_1(y))\,=\,x*p_1(y)\,=\,x\lhd y.\]\end{proof}  

We can use Proposition  \ref{two_augmented_racks} to define a crossed modules of racks in a 
geometrical/topological
context. Before we come to our construction, we recall the fundamental rack of a link from Fenn and Rourke
\cite{FenRou} p. 358.

A {\it link} is a codimension two embedding $L:M\subset Q$ of manifolds. We will assume that $M$ is 
non-empty, that $Q$ is connected (with empty boundary) and that $M$ is transversely oriented in $Q$.
In other words, we assume that each normal disc to $M$ in $Q$ has an orientation which is locally
and globally coherent. 

The link is called {\it framed} if there is a cross-section $\lambda:M\to\partial N(M)$ of the 
normal disk bundle. Denote by $M^+$ the image of $M$ under $\lambda$. In the following, we will only
consider framed links.   

Then, Fenn and Rourke 
associate to $L\subset Q$ an augmented rack (called the {\it fundamental rack} of the link
$L$) which is the space $\Gamma$ of homotopy classes of paths in 
$Q_0:={\rm closure}(Q\setminus N(L))$ of $L$, from a point in $M^+$ to some base point $q_0$. 
During the homotopy, the final point of the path at $q_0$ is kept fixed and the initial
point is allowed to wander at will on $M^+$.

The set $\Gamma$ has an action of the fundamental group $\pi_1(Q_0,q_0)$ defined as follows: 
let $\gamma$ 
be a loop in $Q_0$ based at $q_0$ representing an element $g\in\pi_1(Q_0)$. 
If $\alpha\in\Gamma$ is represented by the 
path $\alpha$, define $a\cdot g$ to be the class of the composite path $\alpha\circ\gamma$.

We can use this action to define a rack structure on $\Gamma$. Let $p\in M^+$ be a point on the 
framing image. Then $p$ lies on a unique meridian circle of the normal disc bundle. Let $m_p$ be the
loop based at $p$ which follows the meridian around in a positive direction. Let $a,b\in\Gamma$ be
represented by paths $\alpha,\beta$ respectively. Let $\partial(b)$ be the element of $\pi_1(Q_0,q_0)$
determined by the homotopy class of the loop $\beta^{-1}\circ m_{\beta(0)}\circ\beta$. The 
\emph{fundamental rack of the framed link $L$} is defined to be the set $\Gamma=\Gamma(L)$ with the operation
\[a\lhd b\,:=\,a\cdot\partial(b):\,=\,[\alpha\circ\beta^{-1}\circ m_{\beta(0)}\circ\beta].\]  

In case the link is evident, but there are different manifolds,
we will denote $\Gamma$ more precisely by $\Gamma_Q$.

Fenn and Rourke show in \cite{FenRou} Proposition 3.1, p. 359, that $\Gamma$ is indeed a rack,
and go on to show that $\partial:\Gamma\to\pi_1(Q_0,q_0)$ is an augmented rack. Furthermore, 
adding in part of the exact homotopy sequence
\[\pi_2(Q)\to \pi_2(Q,Q_0)\to\pi_1(Q_0)\to\pi_1(Q)\]
they show in Proposition 3.2, p. 360, that the associated crossed module of groups 
(using Proposition \ref{augm_racks_to_cr_mod_grps}) of the augmented rack 
$\partial:\Gamma\to\pi_1(Q_0,q_0)$ is Whitehead's crossed module of groups
\[\pi_2(Q,Q_0)\to\pi_1(Q_0).\]

We will now extend this theory to coverings on the topological side and crossed modules of 
augmented racks on the algebraic side.
 
Consider a covering space $\pi:P\to Q$, and a link $L\subset Q$.
Let us furthermore consider a link $L:M\subset Q$ (which we also denote by $L_Q$) and its inverse 
image $L_P:=\pi^{-1}(M)\subset P$. We will suppose that the link $L_p$ is also framed, and this is in
a manner which is compatible with the framing of $L_Q$. 

We have therefore two augmented racks 
\[\Gamma_P\to\pi_1(P_0,p_0)\,\,\,\,\,{\rm and}\,\,\,\,\,\, \Gamma_Q\to\pi_1(Q_0,\pi(p_0)).\]
From now on, we will suppress the base points in the notation.

\begin{theo}   \label{covering_to_cr_mod_racks}
There is an action of $\pi_1(Q_0)$ on $\Gamma_P$ such that the conditions of Proposition 
\ref{two_augmented_racks} are satisfied, i.e. the induced map $\pi:\Gamma_P\to\Gamma_Q$ is a crossed 
module of (augmented) racks.
\end{theo}

\begin{proof} The action is given by the following procedure:

An element $c$ of $\pi_1(Q_0)$ is represented by 
a based loop $\gamma$ in $Q_0$. It lifts to a unique path $\tilde{\gamma}$ in $P_0$ which ends at
the base point $p_0\in P_0$. Now take an element $x$ in $\Gamma_P$, represented by a path $\xi$.
This path $\xi$ projects to a path $\pi(\xi)$ in $Q_0$ which can be lifted to $\widetilde{\pi(\xi)}$ 
such that $\widetilde{\pi(\xi)}(1)=\tilde{\gamma}(0)$ meaning that they are composable. 
The outcome is that $\xi$ (or more precisely some lift of $\pi(\xi)$) can be composed with 
$\tilde{\Gamma}$, and the composition is then, by definition, the action of the homotopy class 
$c=[\gamma]$ on $x=[\xi]$.   

Observe that the map $\beta:\pi_1(P_0)\to\pi_1(Q_0)$ is induced by $\pi:P\to Q$ and is injective. We
identify via $\beta$ the group $\pi_1(P_0)$ as a subgroup of $\pi_1(Q_0)$, the subgroup of loops
in $Q_0$ which lift to loops in $P_0$. 

We therefore clearly have Property {\it 1}, because when lifting the element $\beta(g)$ to a loop in $P_0$,
the action on $x$ becomes $\xi\circ\beta(g)$, which is the action 
in the augmented crossed module $\Gamma_P\to\pi_1(P_0)$.   

The map $\alpha$ (also induced by $\pi:P\to Q$) is clearly equivariant, because $\pi$ distributes on 
the factors of the composition of paths. 

Finally Property {\it 3} is illustrated by the following picture: 

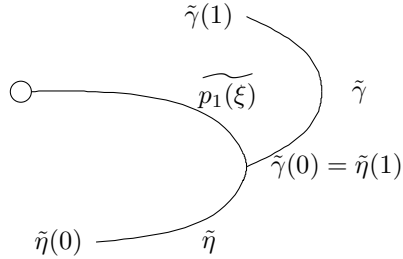
\vspace{.5cm}
\begin{center}
\input{property3.tex}
\end{center}
\vspace{.5cm}

There are two elements $x,y$ of 
$\Gamma_P$ and one element $g$ of $\pi_1(Q_0)$ in play here. Therefore we have two paths $\xi,\eta$
upstairs, and one loop $\gamma$ downstairs. The paths $\xi$ and $\eta$ are pushed down using $\pi$ and 
then lifted to $\tilde{\xi}$ and $\tilde{\eta}$.   

On the LHS of the equation, which is Property {\it 3}, there is an action of $p_1(x)$ applied to the 
element $y$, 
meaning the two paths rejoin each other at the point $p:=\tilde{\gamma}(0)\in P$ in the above
picture, where one of them is carrying a loop at its left end (which is the $\partial=p_1$ map!). 
Moreover, the loop $\gamma$ is lifted to some path $\tilde{\gamma}$ from $\tilde{\gamma}(0)$ 
to the base point $p_0=\tilde{\gamma}(1)$.

On the RHS of the equation, translating into paths in $P_0$, we have two paths from somewhere to $p$ 
and then to $p_0$,
which illustrate the action of $\gamma$ on $x$ and $y$. But then one takes $p_1=\partial$ of the path 
corresponding 
to the action on $x$, which means this path gets a little loop on its left end. Then compose the two 
paths. 

One sees that both sides are equal because on the right hand side, the supplementary round trip along the
path corresponding to the lift of $\gamma$ cancels in the composition.\end{proof}

Observe that it follows from the constructions in this section
that in the above situation, $\pi:P\to Q$ induces simultaneously a map of 
crossed modules of groupes 
\[\xymatrix{ \pi_2(P,P_0)  \ar[r] \ar[d]^{\alpha} & \pi_1(P_0) \ar[d]^{\beta}  \\
             \pi_2(Q,Q_0) \ar[r] & \pi_1(Q_0)  }\]
between the associated crossed modules introduced by Whitehead and a crossed module
$\pi_2(P,P_0)\to\pi_2(Q,Q_0)$ as associated to the crossed module of racks $\Gamma_P\to\Gamma_Q$. 
This follows immediately from the fact that the functor 
${\rm As}:{\tt Racks}\to{\tt Groups}$ sends the fundamental racks associated 
to the links to the corresponding homotopy groups, see \cite{FenRou} 
Proposition 3.2, p. 360.

\end{document}

%% file: property3.tex
\begin{tikzpicture}[baseline=(current bounding box.center)]
\useasboundingbox (-0.5,-0.5) rectangle (5.5,3.5);
\draw[] (0.00,2.00) -- (0.05,2.00) -- (0.10,2.00) -- (0.15,2.00) -- (0.19,2.00) -- (0.24,2.01) -- (0.29,2.01) -- (0.34,2.01) -- (0.39,2.01) -- (0.44,2.01) -- (0.49,2.01) -- (0.54,2.01) -- (0.59,2.01) -- (0.64,2.01) -- (0.69,2.01) -- (0.74,2.01) -- (0.79,2.01) -- (0.84,2.01) -- (0.89,2.00) -- (0.95,2.00) -- (1.00,2.00);
\draw[] (1.00,2.00) -- (1.11,1.99) -- (1.23,1.99) -- (1.35,1.98) -- (1.46,1.97) -- (1.58,1.95) -- (1.70,1.93) -- (1.81,1.91) -- (1.93,1.89) -- (2.04,1.86) -- (2.15,1.83) -- (2.25,1.79) -- (2.35,1.75) -- (2.45,1.70) -- (2.54,1.65) -- (2.63,1.59) -- (2.71,1.52) -- (2.78,1.45) -- (2.84,1.37) -- (2.90,1.29) -- (2.95,1.19);
\draw[] (2.95,1.19) -- (2.95,1.18) -- (2.96,1.17) -- (2.96,1.17) -- (2.97,1.16) -- (2.97,1.15) -- (2.97,1.14) -- (2.98,1.13) -- (2.98,1.12) -- (2.98,1.11) -- (2.99,1.10) -- (2.99,1.09) -- (2.99,1.08) -- (2.99,1.07) -- (2.99,1.06) -- (3.00,1.05) -- (3.00,1.04) -- (3.00,1.03) -- (3.00,1.02) -- (3.00,1.01) -- (3.00,1.00);
\draw[] (3.00,1.00) -- (3.00,0.99) -- (3.00,0.98) -- (3.00,0.97) -- (3.00,0.96) -- (3.00,0.95) -- (2.99,0.94) -- (2.99,0.93) -- (2.99,0.92) -- (2.99,0.91) -- (2.99,0.90) -- (2.98,0.89) -- (2.98,0.88) -- (2.98,0.87) -- (2.97,0.86) -- (2.97,0.85) -- (2.97,0.84) -- (2.96,0.83) -- (2.96,0.82) -- (2.95,0.82) -- (2.95,0.81);
\draw[] (2.95,0.81) -- (2.90,0.71) -- (2.85,0.63) -- (2.78,0.55) -- (2.71,0.48) -- (2.64,0.42) -- (2.56,0.36) -- (2.47,0.31) -- (2.38,0.27) -- (2.28,0.23) -- (2.18,0.19) -- (2.07,0.16) -- (1.96,0.13) -- (1.85,0.11) -- (1.73,0.09) -- (1.61,0.07) -- (1.49,0.05) -- (1.37,0.04) -- (1.25,0.02) -- (1.12,0.01) -- (1.00,0.00);
\draw[] (3.00,3.00) -- (3.06,2.97) -- (3.13,2.94) -- (3.19,2.91) -- (3.26,2.88) -- (3.32,2.85) -- (3.38,2.81) -- (3.44,2.78) -- (3.50,2.74) -- (3.56,2.71) -- (3.61,2.67) -- (3.66,2.63) -- (3.71,2.59) -- (3.76,2.55) -- (3.80,2.51) -- (3.84,2.46) -- (3.87,2.41) -- (3.90,2.36) -- (3.93,2.31) -- (3.96,2.26) -- (3.97,2.20);
\draw[] (3.97,2.20) -- (3.98,2.19) -- (3.98,2.18) -- (3.98,2.17) -- (3.98,2.16) -- (3.99,2.15) -- (3.99,2.14) -- (3.99,2.13) -- (3.99,2.12) -- (3.99,2.11) -- (3.99,2.10) -- (3.99,2.09) -- (4.00,2.08) -- (4.00,2.07) -- (4.00,2.06) -- (4.00,2.05) -- (4.00,2.04) -- (4.00,2.03) -- (4.00,2.02) -- (4.00,2.01) -- (4.00,2.00);
\draw[] (4.00,2.00) -- (4.00,1.99) -- (4.00,1.98) -- (4.00,1.97) -- (4.00,1.96) -- (4.00,1.95) -- (4.00,1.94) -- (4.00,1.93) -- (4.00,1.92) -- (3.99,1.91) -- (3.99,1.90) -- (3.99,1.89) -- (3.99,1.88) -- (3.99,1.87) -- (3.99,1.86) -- (3.99,1.85) -- (3.98,1.84) -- (3.98,1.83) -- (3.98,1.82) -- (3.98,1.81) -- (3.97,1.80);
\draw[] (3.97,1.80) -- (3.96,1.74) -- (3.93,1.69) -- (3.90,1.64) -- (3.87,1.59) -- (3.84,1.54) -- (3.80,1.49) -- (3.76,1.45) -- (3.71,1.41) -- (3.66,1.37) -- (3.61,1.33) -- (3.56,1.29) -- (3.50,1.26) -- (3.44,1.22) -- (3.38,1.19) -- (3.32,1.15) -- (3.26,1.12) -- (3.19,1.09) -- (3.13,1.06) -- (3.06,1.03) -- (3.00,1.00);
\filldraw[fill=white] (0.00,2.00) ellipse (0.14cm and 0.14cm);
\draw (2.50,3.00) node{$\tilde{\gamma}(1)$};
\draw (1.50,2.00) node{$\hspace{2.5cm}\widetilde{p_1(\xi)}$};
\draw (5.50,2.00) node{$\hspace{-2cm}\tilde{\gamma}$};
\draw (3.50,1.00) node{$\hspace{1.4cm}\tilde{\gamma}(0)=\tilde{\eta}(1)$};
\draw (0.50,0.00) node{$\tilde{\eta}(0)$};
\draw (2.50,0.00) node{$\tilde{\eta}$};
\end{tikzpicture}

%% file: cr_mod2.bbl
\begin{thebibliography}{50}

\bibitem{BL} J. Baez, A. Lauda, Higher-Dimensional Algebra V:  2-Groups. \emph{Theory
Theory and Applications of Categories} \textbf{12} (2004), 423--491

\bibitem{Bro1} R. Brown, P. Higgins,
The equivalence of ω-groupoids and cubical T-complexes.
\emph{Cahiers Topologie G\'eom. Diff\'erentielle} \textbf{22}, no. 4 (1981), 349--370

\bibitem{Bro2}
R. Brown, P. Higgins,
The equivalence of $\infty$-groupoids and crossed complexes
\emph{Cahiers Topologie G\'eom. Diff\'erentielle} \textbf{22}, no. 4 (1981), 371--386

\bibitem{CCES} J. S. Carter, A. Crans, M. Elhamdadi, M. Saito, Unpublished correspondence

\bibitem{CEGN} J. Ceniceros, M. Elhamdadi, M. Green, S. Nelson,
Aumented Biracks and their Homology, {\tt arXiv:1309.1678}

\bibitem{FB} M. Forrester-Barker, Group objects and internal categories, 
available as {\tt math.CT/0212065}

\bibitem{FenRou} R. Fenn, C. Rourke,
Racks and links in codimension two. 
\emph{J. Knot Theory Ramifications} \textbf{1} (1992), no. 4, 343--406

\bibitem{FRS} R. Fenn, C. Rourke, B. Sanderson, Trunks and Classifying Spaces. 
\emph{Appl. Cat. Str.} \textbf{3} (1995), 321--356

\bibitem{FreYet} P.J. Freyd, D.N. Yetter,
Braided compact closed categories with applications to
low-dimensional topology.
\emph{Adv. Math.} \textbf{77}, no. 2 (1989), 156--182 

\bibitem{Kas} C. Kassel, Quantum groups.
Graduate Texts in Mathematics, \textbf{155}. Springer-Verlag, New York, 1995

\bibitem{LodKas} C. Kassel, J.-L. Loday,
Extensions centrales d'alg\`ebres de Lie.
\emph{Ann. Inst. Fourier (Grenoble)} \textbf{32}, no. 4 (1982) 119--142

\bibitem{KM} L. H. Kauffman, J.F. Martins,  Invariants of welded virtual 
knots via crossed module invariants of knotted surfaces. 
\emph{Compos. Math.} \textbf{144} (4) (2008), 1046--1080

\bibitem{Lod}  J.-L. Loday,
Spaces with finitely many nontrivial homotopy groups.
\emph{J. Pure Appl. Algebra} \textbf{24} (1982), no. 2, 179--202

\bibitem{Mac} S. Mac Lane, 
Categories for the working mathematician. 
Second edition. Graduate Texts in Mathematics \textbf{5}. Springer-Verlag, New York, 1998


\bibitem{MacWhi} S. MacLane, J.H.C. Whitehead,
On the 3-type of a complex.
\emph{Proc. Nat. Acad. Sci. U.S.A.} \textbf{36} (1950), 41--48


\bibitem{Martins1} J.F. Martins, Categorical Groups, Knots and Knotted Surfaces.  
\emph{J. Knot Theory Ramifications} \textbf{16} No. 9 (2007),  1181--1217

\bibitem{Martins} J.F. Martins, The Fundamental Crossed Module of the 
Complement of a Knotted Surface, \emph{Trans. Amer. Math. Soc.} \textbf{361} 
(2009), 4593--4630   

\bibitem{MarPic} J. F. Martins, R. Picken, 
Link invariants from finite categorical groups and 
braided crossed modules. {\tt arXiv:1301.3803}

\bibitem{Wag} F. Wagemann, On crossed modules.
\emph{Comm. Algebra} \textbf{34}, no. 5 (2006), 1699--1722  

\bibitem{W1} J.H.C. Whitehead, Note on a previous paper 
entitled ``On adding relations to homotopy groups'', 
\emph{Ann. Math.} \textbf{47} (1946), 806--810

\bibitem{W2} J.H.C. Whitehead, Combinatorial homotopy II. 
\emph{Bull. Amer. Math. Soc.} \textbf{55} (1949), 453--496




\end{thebibliography}
